\documentclass[twoside,11pt,reqno,a4paper]{amsart}

\usepackage{amsmath,amssymb,amscd,mathrsfs,epic,latexsym,tikz,mathrsfs,cite,hyperref}
\usepackage[matrix,arrow]{xy}
\usepackage{wasysym,enumerate}
\usetikzlibrary{decorations.markings}

\newcommand{\arxiv}[1]{\href{http://arxiv.org/abs/#1}{\tt arXiv:\nolinkurl{#1}}}
\newcommand{\arXiv}[1]{\href{http://arxiv.org/abs/#1}{\tt arXiv:\nolinkurl{#1}}}

\numberwithin{equation}{section}

\newtheorem{Proposition}[equation]{Proposition}
\newtheorem{Lemma}[equation]{Lemma}
\newtheorem{Theorem}[equation]{Theorem}
\newtheorem{Corollary}[equation]{Corollary}

\theoremstyle{definition}  

\let\<\langle
\let\>\rangle

\newcommand\Comment[2][\relax]{\space\par\medskip\noindent%
   \fbox{\begin{minipage}{\textwidth}\textbf{Comment\ifx\relax#1\else---#1\fi}\newline%
        #2\end{minipage}}\medskip
}

\def\bi{\text{\boldmath$i$}}
\def\bj{\text{\boldmath$j$}}

\def\b1{\text{\boldmath$1$}}
\def\pmod#1{\text{ }(\text{\rm mod } #1)\,}

\newcommand{\tw}{\operatorname{s}}
\newcommand{\Pol}{{P}}

\newcommand{\End}{\operatorname{End}}

\newcommand{\im}{\operatorname{im}}

\newcommand{\res}{\operatorname{res}}

\newcommand{\Stab}{\operatorname{Stab}}
\newcommand{\op}{\operatorname{op}}

\newcommand{\words}{\langle I\rangle}

\newcommand{\Z}{\mathbb{Z}}

\def\phi{{\varphi}}

\newcommand{\ga}{\gamma}
\newcommand{\Ga}{\Gamma}
\newcommand{\la}{\lambda}
\newcommand{\La}{\Lambda}
\newcommand{\al}{\alpha}
\newcommand{\be}{\beta}

\def\Si{\mathfrak{S}}
\newcommand{\si}{\sigma}

\newcommand{\de}{\delta}
\newcommand{\De}{\Delta}

\newcommand{\Ind}{{\mathrm {Ind}}}

\newcommand{\Sym}{{\La}}

\newcommand{\A}{{\mathscr A}}

\newcommand{\shift}{{\tt sh}}

\newcommand\Par{{\mathscr P}}
\newcommand\StTab{{\mathscr T}}

\def\T{{\mathtt T}}

\def\Stab{{\mathtt S}}

\def\spa{\operatorname{span}}

\def\into{{\hookrightarrow}}

\def\iso{\stackrel{\sim}{\longrightarrow}}

\colorlet{darkgreen}{green!50!black}
\tikzset{dots/.style={very thick,loosely dotted},
         greendot/.style={fill,circle,color=darkgreen,inner sep=1.5pt,outer sep=0}
}
\def\greendot(#1,#2){\node[greendot] at(#1,#2){}}

\newenvironment{braid}{
  \begin{tikzpicture}[baseline=6mm,blue,line width=1pt, scale=0.4,
                      draw/.append style={rounded corners},
                      every node/.append style={font=\tiny}]%
  }{\end{tikzpicture}
}

\author[Alexander Kleshchev]{\sc Alexander S. Kleshchev}
\address{Department of Mathematics\\ University of Oregon\\
Eugene\\ OR~97403, USA}
\email{klesh@uoregon.edu}

\author[Joseph Loubert]{\sc Joseph W. Loubert}
\address{Department of Mathematics\\ University of Oregon\\
Eugene\\ OR~97403, USA}
\email{loubert@uoregon.edu}

\author[Vanessa Miemietz]{\sc Vanessa Miemietz}
\address{School of Mathematics\\ University of East Anglia\\
Norwich\\ NR4 7TJ, UK}
\email{v.miemietz@uea.ac.uk}

\thanks{Research supported in part by the NSF grant no. DMS-1161094, the Humboldt Foundation, and the ERC grant PERG07-GA-2010-268109. The paper has been completed at the University of Stuttgart. The authors thank Steffen Koenig for hospitality.}

\begin{document}

\begin{abstract}
We prove that the Khovanov-Lauda-Rouquier algebras $R_\al$ of type $A_\infty$ are (graded) affine cellular in the sense of Koenig and Xi. In fact, we establish a stronger property, namely that the affine cell ideals in $R_\al$ are generated by idempotents.  This in particular implies the (known) result that the global dimension of $R_\al$ is finite, and yields a theory of standard and reduced standard modules for $R_\al$. 
\end{abstract}

\title[Affine Cellularity of KLR algebras in type A]{Affine Cellularity of Khovanov-Lauda-Rouquier algebras in type A}

\maketitle

\section{Introduction}
The goal of this paper is to establish (graded) affine cellularity for the Khovanov-Lauda-Rouquier algebras $R_\al$ of type $A_\infty$ in the sense of Koenig and Xi \cite{KoXi}. In fact, we construct a chain of affine cell ideals in $R_\al$ which are generated by idempotents. This stronger property is analogous to quasi-heredity for finite dimensional algebras, and by a general result of Koenig and Xi \cite[Theorem 4.4]{KoXi}, it also implies finite global dimension of $R_\al$. Thus we obtain a new proof of a recent result of Kato \cite{Kato} and McNamara \cite{McN} in type $A$ over an arbitrary field. As another application, we automatically get a theory of standard and reduced standard modules, cf. \cite{Kato}.  

The (finite dimensional) cyclotomic quotients of $R_\al$ have been shown to be graded cellular by Hu and Mathas \cite{HM}. Their proof uses the isomorphism theorem from \cite{BKyoung}, the ungraded cellular structure constructed in \cite{DJM}, and the seminormal forms of cyclotomic Hecke algebras. The affine cellular structure that we construct here is combinatorially less intricate and does not appeal to seminormal forms. 

Our affine cellular basis is built from scratch, using only the defining relations, some weight theory from \cite{KR}, and a dimension formula \cite[Theorem 4.20]{BKgrdec}. It is not clear whether it can be deduced from the basis in \cite{HM} by a limiting procedure. At any rate, our philosophy is that one should first construct affine cellular structures and then `project' them to the quotients. This seems to be the only approach available for Lie types other than $A$. 

 We now give a definition of (graded) affine cellular algebra from \cite[Definition 2.1]{KoXi}. For this introduction, we fix a noetherian domain $k$ (later on it will be sufficient to work with $k=\Z$). 
 By definition, an {\em affine algebra} is a quotient of a polynomial algebra $k[x_1,\dots,x_n]$ for some $n$. 

Throughout the paper,  unless otherwise stated, we assume that all algebras are ($\Z$)-graded, all ideals, subspaces, etc. are homogeneous, and all homomorphisms are homogeneous degree zero homomorphisms with respect to the given gradings.

 Let $A$ be a (graded) unital $k$-algebra with a $k$-anti-involution $\tau$. A (two-sided) ideal $J$ in $A$ is called an \emph{affine cell ideal} if the following conditions are satisfied: 
\begin{enumerate}
\item $\tau(J) = J$;
\item there exists an affine $k$-algebra $B$ with a $k$-involution $\si$ and a free $k$-module $V$ of finite rank such that $\Delta:=V \otimes_k B$ has an $A$-$B$-bimodule structure, with the right $B$-module structure induced by the regular right $B$-module structure on $B$;
\item let $\Delta' := B \otimes_k V$ be the $B$-$A$-bimodule with left $B$-module structure induced by the regular left $B$-module structure on $B$ and right $A$-module structure defined by 
\begin{equation}\label{EAction}
(b\otimes v)a = \tw(\tau(a)(v \otimes b)),
\end{equation} 
where 
$\tw:V\otimes_k B\to B\otimes_k V,\ v\otimes b\to b\otimes v$; 
then there is an $A$-$A$-bimodule isomorphism $\alpha: J \to \Delta \otimes_B\Delta'$, such that the following diagram commutes:

$$\xymatrix{ J \ar^-{\alpha}[rr]  \ar^{\tau}[d]&& \Delta \otimes_B\Delta' \ar^{v \otimes b \otimes b' \otimes w \mapsto w \otimes \si(b') \otimes \si(b) \otimes v }[d] \\ J \ar^-{\alpha}[rr]&&\Delta \otimes_B\Delta'.}$$
\end{enumerate}
The algebra $A$ is called \emph{graded affine cellular} if there is a $k$-module decomposition $A= J_1' \oplus J_2' \oplus \cdots \oplus J_n'$ 
with $\tau(J_l')=J_l'$ for $1 \leq l \leq n$, such that, setting $J_m:= \bigoplus_{l=1}^m J_l'$, we obtain an ideal filtration
$$0=J_0 \subset J_1 \subset J_2 \subset \cdots \subset J_n=A$$
so that each $J_m/J_{m-1}$ is an affine cell ideal of $A/J_{m-1}$. 

To describe our main results we introduce some notation referring the reader to the main body of the paper for details. Let $Q_+$ be the non-negative root lattice corresponding to the root system of type $A_\infty$, $\al\in Q_+$ of height $d$, and $R_\al$ be the corresponding KLR algebra with standard generators $e(\bi), \psi_1,\dots,\psi_{d-1},y_1,\dots,y_d$. We denote by
$\Pi(\al)$ be the set of root partitions of $\al$, see Section~\ref{SSRP}. To any $\pi\in\Pi(\al)$ we associate the Young subgroup $\Si_\pi\leq\Si_d$ and denote by $\Si^\pi$ the set of the shortest left coset representatives for $\Si_\pi$ in $\Si_d$. We define the polynomial subalgebras $\La_\pi\subseteq R_\al$ -- these are isomorphic to tensor products of algebras of symmetric polynomials, see (\ref{ELaPi}). We also define the monomials $y_\pi\in R_\al$ and idempotents $e_\pi\in R_\al$, see Section~\ref{SSBasic}. Then we set
\begin{align*}
I_\pi' &:= k\text{-}\spa\{\psi_w y_\pi \Sym_\pi e_\pi  \psi_v^\tau \ |\ w, v \in \Si^\pi\},
\end{align*}
$I_\pi := \sum_{\si \geq \pi}{I_\si'}$, and $I_{>\pi} = \sum_{\si > \pi}{I_\si'}$. 
Our main results are now as follows:

\vspace{2mm}
\noindent
{\bf Main Theorem.}
{\em
The algebra $R_\al$ is graded affine cellular with cell chain given by the ideals $\{I_\pi\mid \pi \in \Pi(\al)\}$. Moreover, setting $\bar R_\al := R_\al / I_{>\pi}$ for a fixed $\pi\in \Pi(\al)$, we have:
  \begin{enumerate}
  \item[{\rm (i)}] the map $\Sym_\pi \to \bar e_\pi \bar R_\al \bar e_\pi,\ b\mapsto \bar b\bar e_\pi$ is an isomorphism of graded algebras;
  \item[{\rm (ii)}] $\bar R_\al \bar e_\pi$ is a free right $\bar e_\pi \bar R_\al \bar e_\pi$-module with basis $\{\bar \psi_w \bar y_\pi \bar e_\pi\ |\ w \in \Si^\pi\}$;
  \item[{\rm (iii)}] $\bar e_\pi \bar R_\al $ is a free left $\bar e_\pi \bar R_\al \bar e_\pi$-module with basis $\{\bar e_\pi \bar \psi_v^\tau\ |\ v \in \Si^\pi\}$;
  \item[{\rm (iv)}] multiplication provides an isomorphism
      \[
        \bar R_\al \bar e_\pi \otimes_{\bar e_\pi \bar R_\al \bar e_\pi} \bar e_\pi \bar R_\al \iso \bar R_\al \bar e_\pi \bar R_\al;
      \]
\item[{\rm (v)}]  $\bar R_\al \bar e_\pi \bar R_\al = I_\pi / I_{>\pi}$.
  \end{enumerate}
}

\vspace{2mm}
Main Theorem(v) shows that each affine cell ideal $I_\pi/I_{>\pi}$ in $A/I_{>\pi}$ is generated by an idempotent. This, together with the fact that each algebra $\La_\pi$ is a polynomial algebra, is enough to invoke \cite[Theorem 4.4]{KoXi} to get

\vspace{2 mm}
\noindent
{\bf Corollary.}
{\em 
If the ground ring $k$ has finite global dimension, then the algebra $R_\al$ has finite global dimension. 
}

\vspace{2 mm}
This seems to be a slight generalization of \cite{Kato} and \cite{McN} (in type $A$ only) in two ways: Kato works over fields of characteristic zero, and McNamara seems to work over arbitrary fields; moreover, \cite{Kato} and \cite{McN} deal with categories of graded modules only, while our corollary holds for the algebra $R_\al$ even as an ungraded algebra. 

In the following conjectures we use the term {\em graded affine quasi-hereditary}\, to denote the graded affine cellular algebras with the affine cell ideals satisfying the additional nice properties described in Main Theorem.

\vspace{2mm}
\noindent
{\bf Conjecture.}
{\em
{\rm (i)} All Khovanov-Lauda-Rouquier algebras are graded affine cellular. 

{\rm (ii)} All cyclotomic  Khovanov-Lauda-Rouquier algebras are graded cellular.

{\rm (iii)} Let us fix a Lie type $\Gamma$. Then the Khovanov-Lauda-Rouquier algebras $R_\al(\Ga)$ are graded affine quasi-hereditary for all $\al\in Q_+$ if and only if $\Ga$ is of finite type. 
}

\vspace{2 mm}
In \cite{KLade}, we prove this conjecture for finite simply laced Lie types $\Ga$. 

The organization of the paper is as follows. Section~\ref{SPrel} is preliminary. Section~\ref{SDim} establishes a graded dimension formula for $R_\al$, which is later used to show that the elements of our affine cellular basis are actually linearly independent. Section~\ref{SNH} deals with the special case of the affine nilHecke algebra. This case will be fed into the proof of the general case. Finally, in Section~\ref{SMain}, we prove the main results.

\section{Preliminaries}\label{SPrel}

\subsection{Lie theoretic notation}

Let $\Gamma$ be the Dynkin quiver of type $A_\infty$ with the set of vertices $I=\Z$ and the corresponding 
Cartan matrix 
\begin{equation}\label{Cartan}
a_{i,j} := \left\{
\begin{array}{rl}
2&\text{if $i=j$},\\
0&\text{if $|i-j|>1$},\\
-1&\text{if $i=j \pm 1$}
\end{array}\right.
\end{equation}
for $i,j \in I$. We have a set of simple roots $\{\al_i \mid i\in I\}$ and the positive part of the root lattice $Q_+ := \bigoplus_{i \in I} \Z_{\geq 0} \al_i$. The set of positive roots is given by
$$\{\al(m,n):=\al_m+\al_{m+1}+\dots+\al_n\mid  m,n \in I, \ m\leq n\}.$$
For $\al=\sum_{i \in I} c_i \al_i \in Q_+$, we denote by $|\al|:=\sum_{i \in I} c_i$ the height of $\alpha$. We furthermore have a set of fundamental weights $\{ \omega_i \mid i \in I\}$ 
and the set of dominant weights $P_+:=\bigoplus_{i \in I} \Z_{\geq 0} \omega_i$.



The symmetric group $\Si_d$ with basic transpositions $s_1,\dots,s_{d-1}$ acts on the set $I^d$ by place permutation. 
The orbits  are the sets
$$
\words_\al := \{\bi = (i_1,\dots,i_d)\in I^d\:|\:\al_{i_1}+\cdots+\al_{i_d}=\alpha\}
$$
for each $\al \in Q_+$ with $|\al|=d$. 
We let $\geq$ denote the lexicographic order on $\words_\al$ determined by the natural order on $I=\Z$.

To a positive root $\be = \al(m,n)$, we associate the word $$\bi_\be:= (m,m+1,\dots,n) \in\words_\be.$$ We denote the set of positive roots by $\Phi_+$. 
Define a total order on $\Phi_+$ by 
\begin{equation}\label{EOrderPhi}
\be\leq \ga\ \text{if and only if}\ \bi_\be\leq\bi_\ga\qquad(\be,\ga\in\Phi_+).
\end{equation}

\subsection{KLR Algebras}
For $\alpha\in Q_+$ of height $d$ and the commutative unital ground ring $k$, let $R_\alpha=R_\al(k)$ denote the associative, unital $k$-algebra on generators
$\{e(\bi)\:|\:\bi \in \words_\alpha\}\cup\{y_1,\dots,y_d\}\cup\{\psi_1,\dots,\psi_{d-1}\}$
subject to the following relations 
\begin{align*}
&e(\bi) e(\bj) = \de_{\bi,\bj} e(\bi);  \qquad\sum_{\bi \in \words_\al} e(\bi) = 1;\\
&y_r e(\bi) = e(\bi) y_r; \qquad
\psi_r e(\bi) = e(s_r\cdot\bi) \psi_r;\quad y_r y_s = y_s y_r;\\
&\psi_r y_s  = y_s \psi_r \qquad\text{if $s \neq r,r+1$};\\
&\psi_r \psi_s = \psi_s \psi_r\qquad\text{if $|r-s|>1$};\\
&\psi_r y_{r+1} e(\bi) = (y_r\psi_r+\de_{i_r,i_{r+1}})e(\bi);\quad 
y_{r+1} \psi_re(\bi) =(\psi_r y_r+\de_{i_r,i_{r+1}}) e(\bi);
\\
&\psi_r^2e(\bi) = 
\left\{\begin{array}{ll}
0&\text{if $i_r = i_{r+1}$},\\
e(\bi)&\text{if $|i_r -i_{r+1}|>1$},\\
(y_{r+1}-y_r)e(\bi)&\text{if $i_r =i_{r+1}+1$},\\
(y_r - y_{r+1})e(\bi)&\text{if $i_r =i_{r+1}-1$;}
\end{array}\right.\\
&\psi_{r}\psi_{r+1} \psi_{r} e(\bi)
=
\left\{\begin{array}{ll}
(\psi_{r+1} \psi_{r} \psi_{r+1} +1)e(\bi)&\text{if $i_{r+2}=i_r =i_{r+1}+1$},\\
(\psi_{r+1} \psi_{r} \psi_{r+1} -1)e(\bi)&\text{if $i_{r+2}=i_r = i_{r+1}-1$},\\
\psi_{r+1} \psi_{r} \psi_{r+1} e(\bi)&\text{otherwise}.
\end{array}\right.
\label{R7}
\end{align*}

There is a unique $\Z$-grading on $R_\alpha$ such that all $e(\bi)$ are of degree $0$,
all $y_r$ are of degree $2$, and $\deg(\psi_r e(\bi))=-a_{i_r,i_{r+1}}$ (see \ref{Cartan}).

Fixing a reduced decomposition $w=s_{r_1}\dots s_{r_m}$ for each $w\in \Si_d$, we define the elements 
$
\psi_w:=\psi_{r_1}\dots\psi_{r_m}\in R_\al
$
for all $w\in \Si_d$. 

\begin{Theorem}\label{TBasisGen}
{\cite[Theorem 2.5]{KL1}}, \cite[Theorem 3.7]{R} 
A $k$-basis of $R_\al$ is given by
$$ \{\psi_w y_1^{m_1}\dots y_d^{m_d}e(\bi)\mid w\in \Si_d,\ m_1,\dots,m_d\in\Z_{\geq 0}, \ \bi\in \words_\al\}.
$$ 
\end{Theorem}

The commutative subalgebra of $R_\al$ generated by $y_1,\dots,y_d$ is thus isomorphic to the polynomial algebra $k[y_1,\dots,y_d]$ and will be denoted by $\Pol_d$. In view of Theorem~\ref{TBasisGen}, we have $R_\al(k)\simeq R_\al(\Z)\otimes_\Z k$, so in what follows we can work with $k=\Z$. When we need to deal with representation theory of $R_\al$ we often  assume that $k$ is a field.

We will also use the diagrammatic notation introduced in~\cite{KL1} to represent elements of $R_\al$. Given $\bi=(i_1, \dots, i_d) \in \words_\al$, we write
\[
e(\bi)=
\begin{braid}\tikzset{baseline=7mm}
  \draw (0,4)node[above]{$i_1$}--(0,0);
  \draw (1,4)node[above]{$i_2$}--(1,0);
  \draw[dots] (1.2,4)--(3.8,4);
  \draw[dots] (1.2,0)--(3.8,0);
  \draw (4,4)node[above]{$i_d$}--(4,0);
\end{braid}\hspace{-1.5mm},
\ \psi_r e(\bi)=
\begin{braid}\tikzset{baseline=7mm}
  \draw (0,4)node[above]{$i_1$}--(0,0);
  \draw[dots] (0.2,4)--(1.8,4);
  \draw[dots] (0.2,0)--(1.8,0);
  \draw (2,4)node[above]{$i_{r-1}$}--(2,0);
  \draw (3,4)node[above]{$i_r$}--(4,0);
  \draw (4,4)node[above]{$i_{r+1}$}--(3,0);
  \draw (5,4)node[above]{}--(5,0);
  \draw[dots] (5.2,4)--(6.8,4);
  \draw[dots] (5.2,0)--(6.8,0);
  \draw (7,4)node[above]{$i_d$}--(7,0);
\end{braid}\hspace{-1.5mm},
\ 
y_s e(\bi) =
\begin{braid}\tikzset{baseline=7mm}
  \draw (0,4)node[above]{$i_1$}--(0,0);
  \draw[dots] (0.2,4)--(1.8,4);
  \draw[dots] (0.2,0)--(1.8,0);
  \draw (2,4)node[above]{$i_{s-1}$}--(2,0);
  \draw (3,4)node[above]{$i_s$}--(3,0);
  \greendot(3,2);
  \draw (4,4)node[above]{$i_{s+1}$}--(4,0);
  \draw[dots] (4.2,4)--(5.8,4);
  \draw[dots] (4.2,0)--(5.8,0);
  \draw (6,4)node[above]{$i_{d}$}--(6,0);
\end{braid}
\]
where \(1\le r<d\) and $1\le s\le d$. 

We say that $w\in \Si_d$ possesses a \emph{left-right symmetric reduced decomposition}, if we can write $w= s_{r_1}\dots s_{r_m}=s_{r_m}\dots s_{r_1}$ and the reduced expression $s_{r_m}\dots s_{r_1}$ can be obtained from $s_{r_1}\dots s_{r_m}$ by using only commuting Coxeter relations.

\begin{Lemma} \label{LLR}
  The element $w_0 \in \Si_d$ possesses a left-right symmetric reduced decomposition.
\end{Lemma}
\begin{proof}
Since reversing the order of the basic transpositions in a reduced decomposition amounts to reflecting the representing braid diagram across the horizontal, it is easy to see that 
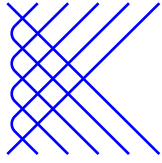
\[
\begin{braid}\tikzset{baseline=7mm}
  \draw (0,5)--(5,0);
  \draw (1,5)--(0,4)--(4,0);
  \draw (2,5)--(0,3)--(3,0);
  \draw (3,5)--(0,2)--(2,0);
  \draw (4,5)--(0,1)--(1,0);
  \draw (5,5)--(0,0);
\end{braid}
\]
gives rise to a left-right symmetric reduced decomposition as desired.
\end{proof}

\subsection{Root partitions}\label{SSRP}
Let $\al\in Q_+$. A {\em root partition of $\al$} is a way to write $\al$ as an ordered sum of positive roots 
$
\al=p_1\be_1+\dots+p_N\be_N
$
so that $\be_1>\dots>\be_N$ and $p_1,\dots,p_N>0$. We denote such a root partition $\pi$ as follows:
\begin{equation}\label{ERP}
\pi=\be_1^{p_1}\dots \be_N^{p_N}.
\end{equation}
The set of all root partitions of $\al$ is denoted $\Pi(\al)$. 

To a root partition  $\pi$ as in \eqref{ERP} we associate the word $$\bi_\pi:= \bi_{\be_1}\dots\bi_{\be_1}\dots\bi_{\be_N}\dots\bi_{\be_N} \in \words_\al$$
as the concatenation of the $\bi_{\be_k}$ where each $\bi_{\be_k}$ occurs $p_k$ times.
Define the total order on $\Pi(\al)$ via  $\pi\geq\si$ if and only if $\bi_\pi\geq\bi_\si$ for $\pi,\si \in \Pi(\al)$.

To a root partition $\pi$ as in (\ref{ERP}) we also associate a parabolic subgroup 
$\Si_\pi\leq \Si_d$:
$$
\Si_\pi\cong \Si_{|\be_1|}^{\times p_1}\times\dots\times \Si_{|\be_N|}^{\times p_N}
$$ 
and the set $\Si^\pi$ of the minimal length left coset representatives of $\Si_\pi$ in $\Si_d$. The orbits of $\Si_\pi$ on $\{1,\dots,d\}$ will be referred to as {\em $\pi$-blocks}. The first $p_1$ of the $\pi$-blocks are of size $|\be_1|$, and will be referred to as the {\em $\pi$-blocks of weight $\be_1$}, the next $p_2$ of the $\pi$-blocks are of size $|\be_2|$, and will be referred to as the {\em $\pi$-blocks of weight $\be_2$}, etc.

\subsection{Representation theory}
Set $\A:=\Z[q,q^{-1}]$. 
For a graded vector space $V=\oplus_{n\in \Z}V_n$ we 
set $\dim_q V:= \sum_{n \in \Z} (\dim V_n) q^n$. If $f=\sum_{m\in \Z}a_mq^m\in\A$,  we denote $\deg_n(f):=a_nq^n$. 
We denote by $V\langle m \rangle$ the graded vector space with degrees shifted up by $m$ so that $V\langle m \rangle_n = V_{n-m}$.

We will use an operation of induction on the KLR-algebras defined in \cite{KL1}. Given an $R_\al$-module $M$ and an $R_\be$-module $N$, we thus have an induced module $\Ind_{\al,\be}M\boxtimes N$ over $R_{\al+\be}$, which will also be denoted $M\circ N$.  


For a positive root $\be$, there is a unique one-dimensional  $R_\be$-module $L(\be)$ with $e(\bi_\be)L(\be) \neq 0$ and all other generators acting as zero. For a root partition $\pi \in \Pi(\al)$ as in \eqref{ERP} we set
$
\shift(\pi):=\sum_{k=1}^N p_k(p_k-1)/2
$
and, following \cite[7.1]{KR}, define the \emph{(reduced) standard module}
$$
\bar\De(\pi):= L(\be_1)^{\circ p_1}\circ \cdots \circ L(\be_N)^{\circ p_N}\langle \shift(\pi) \rangle.
$$
Let $k$ be a field. By \cite[Theorem 7.2]{KR}, $\bar\De(\pi)$ has a unique irreducible quotient, denoted by $L(\pi)$, and $\{L(\pi)\mid\pi\in \Pi(\al)\}$ is a  complete system of (graded) irreducible $R_\al$-modules up to isomorphism. Furthermore, $\bi_\pi$ is lexicographically the largest among the words $\bi\in\words_\al$ such that $e(\bi)L(\pi)\neq 0$.

\subsection{Poincar\'e polynomials}
We will make use of the following well-known computation of the Poincar\'e polynomial, see e.g. \cite[Theorem 3.15]{Hum}:

\begin{Lemma}\label{lem:nHdim}
We have 
  \[
    \sum_{w \in \Si_a} t^{\ell(w)} = \prod_{r=1}^a \frac{t^r-1}{t-1} = t^{a(a-1)/2} \prod_{r=1}^a \frac{1-t^{-r}}{1-t^{-1}}.
  \]
\end{Lemma}

\section{A dimension formula}\label{SDim}

In this section we establish a graded dimension formula for $R_\al$. This formula can be thought of as a combinatorial shadow of the affine cellular structure on $R_\al$ to be constructed later. We point out that there is a similar dimension formula for any finite type KLR algebra 
\cite{KLade}. The proof we give here works for type $A$ only, but it might be of independent interest since it exploits  a `limiting procedure' and the dimension formula from \cite[Theorem 4.20]{BKgrdec}. 
By Theorem~\ref{TBasisGen}, the graded dimension of $R_\al(k)$ does not depend on $k$, so in this section we fix $k$ a field.


We start by the following observation:

\begin{Lemma} \label{LThu}
Let $\be=\al_i+\dots+\al_j$ and $\ga=\al_i+\dots+\al_k$ for $j>k$. Then $L(\be)\circ L(\ga)\simeq L(\ga)\circ L(\be)\langle 1\rangle$ is irreducible. 
\end{Lemma}
\begin{proof}
It is easy to see that $\bi_\be\bi_\ga$ is the only dominant weight in $L(\be)\circ L(\ga)$, and it appears with multiplicity one. The result easily follows, cf. \cite{KR}.
\end{proof}

\subsection{Cyclotomic KLR-algebras.}
For the rest of this section we fix $\al\in Q_+$ of height $d$. Let 
\begin{equation}\label{EOm}
\Omega= \sum_{i\in I} b_i\omega_i
\end{equation} 
be a dominant weight of level $l:=\sum_{i\in I} b_i$, and consider the corresponding cyclotomic quotient $R^\Omega_\al$. We will use the notation and results of \cite{BKgrdec,BKW}. In particular,  by $\Par^\Omega_\al$ is the set of all $l$-multipartitions of weight $\al$, cf. \cite[(3.15)]{BKgrdec}, for $\la\in\Par^\Omega_\al$, we denote by $\StTab(\la)$ the set of standard $\la$-tableaux, $\deg(\Stab)$ denotes the degree of the standard tableau $\Stab\in\StTab(\la)$, cf. \cite[Section 4.11]{BKgrdec}\cite[Section 3.2]{BKW}, and $S^\la$ denotes the Specht module corresponding to $\la$, cf. \cite[Section 4.2]{BKW}. The definition of $\deg(\Stab)$ depends on the choice of a multicharge $\kappa=(k_1,\dots,k_l)$ such that $\omega_{k_1}+\dots+\omega_{k_l}=\Omega$, cf. \cite[Section 3]{BKgrdec}. We always make the choice for which $k_1\geq\dots\geq k_l$. 
By \cite[Theorem 4.20]{BKgrdec}, we have
\begin{equation*}\label{EWed1}
\dim_q R^\Omega_\al= \sum_{\la \in \Par^\Omega_\al} \Big(\sum_{\Stab \in \StTab(\la)} q^{\deg \Stab}\Big)^2.
\end{equation*}
By \cite[Corollary 3.14]{BKW}, we can rewrite this as follows:
\begin{equation}\label{EWed2}
\dim_q R^\Omega_\al= \sum_{\la \in \Par^\Omega_\al} \Big(\sum_{\Stab \in \StTab(\la)} q^{\deg(\psi_{w_\Stab}e(\bi^\la))+\deg(\T^\la)}\Big)^2,
\end{equation}
where $\T^\la$ is the leading $\la$-tableau, $\bi^\la\in\words_\al$ is the corresponding residue sequence, and $w_\Stab$ is defined by $w_\Stab\T^\la=\Stab$, cf. \cite[Section 3.2]{BKW}.  

The symmetric group $\Si_l$ acts on $l$-multipartitions by permuting their components, so that $w\cdot\la=(\la^{(w^{-1}(1))},\dots,\la^{(w^{-1}(l))})$. The parabolic subgroup $\Si(\Omega)=\times_{i\in I}\Si_{b_i}\leq \Si_l$ then acts on $\Par^\Omega_\al$. 
Let $\la=(\la^{(1)},\dots,\la^{(l)})\in\Par^\Omega_\al$ be such that each $\la^{(m)}$ is a non-trivial one-row partition. Let $\be_m=\sum_{b\in\la^{(m)}}\al_{\res(b)}$ where the summation is over all boxes $b$ of $\la^{(m)}$ and $\res(b)\in I$ denotes the residue of $b$. There exists an element $w\in\Si_l$ such that $\be_{w^{-1}(1)}\geq\dots\geq \be_{w^{-1}(l)}$. Let $\ell_\la$ be the length of the shortest such element, and define
$$
\pi(\la):=\be_{w^{-1}(1)}\dots \be_{w^{-1}(l)}\in \Pi(\al).
$$

By inflation we consider all $R^\Omega_\al$-modules as $R_\al$-modules. We want to connect  standard modules to some special Specht modules.

\begin{Proposition}\label{PThu}
Let $\la=(\la^{(1)},\dots,\la^{(l)})\in\Par^\Omega_\al$ be such that each $\la^{(m)}$ is a non-trivial one-row partition. Denote $\la^{\op}:=w_0^\Omega\cdot\la$, where $w_0^\Omega$ is the longest element in $\Si(\Omega)$. 
Then 
$
S^{\la^{\op}} \simeq \bar\De(\pi(\la))\< \ell_\la\>
$
\end{Proposition}
\begin{proof}
This follows from Lemma~\ref{LThu} and \cite[Theorem 8.2]{KMR}
\end{proof}

For a root partition $\pi$ as in (\ref{ERP}) and a positive integer $p$, denote
\begin{align*}
c_\pi&:=q^{\shift(\pi)}\sum_{w \in \Si^\pi} q^{\deg \psi_w e(\bi_\pi)},
\\
l_p&:= \prod_{m=1}^{p} \frac{1}{1-q^{2m}},
\\
l_\pi&:=\prod_{k=1}^N l_{p_k}.
\end{align*}
Note that $c_\pi$ is the dimension of the reduced standard module $\bar\De(\pi)$ and $l_p$ is the dimension of the algebra $\La_p$ of symmetric polynomials in $p$ variables of degree  $2$.

\subsection{The formula}
Our dimension formula is now as follows:

\begin{Proposition} \label{PDim}
We have 
$$\dim_q R_\al=\sum_{\pi\in \Pi(\al)}l_\pi c_\pi^2.
$$
\end{Proposition}
\begin{proof}

Let us fix $n\in\Z$. It suffices to prove that $$\deg_n(\dim_q R_\al)=\deg_n(\sum_{\pi\in \Pi(\al)}l_\pi c_\pi^2).$$ Note that we can choose  $b_i\gg0$ for all $i$ in the support of $\al$, such that $\deg_n(\dim_q R_\al)=\deg_n(\dim_q R_\al^\Omega)$. Let us make this choice and prove that $\deg_n(\dim_q R_\al^\Omega)=\deg_n(\sum_{\pi\in \Pi(\al)}l_\pi c_\pi^2)$.

\vspace{1 mm}
\noindent
{\em Claim 1.} Let $\mathscr{R}$ be the set of all multipartitions $\la=(\la^{(1)},\dots,\la^{(l)})\in\Par_\al^\Omega$ such that each $\la^{(a)}$ is either empty or one row. Then
$$
\deg_n(\dim_q R^\Omega_\al)=\deg_n\Big(\sum_{\la \in \mathscr{R}} \Big(\sum_{\Stab \in \StTab(\la)} q^{\deg(\psi_{w_\Stab}e(\bi^\la))+\deg(\T^\la)}\Big)^2\Big).
$$

\noindent
{\em Proof of Claim 1.}
Note that $\deg(\psi_{w_\Stab}e(\bi^\la))\geq -2d!$. So in view of (\ref{EWed2}), it suffices to prove that $\deg(\T^\la)\gg n$ unless $\Stab\in\mathscr{R}$. Let $\la\in\Par_\al^\Omega$. If $\la^{(m)}\neq \emptyset$, then $k_m$ is in the support of $\al$. If $k_m-1$ is not in the support of $\al$, then $\la^{(m)}$ can only have one row. If $k_m-1$ is in the support of $\al$ and $\la$ has at least two rows, then  $\deg(T^\la)\geq b_{k_{m}-1}\gg0$. Claim 1 is proved.

Next, let $\la$ be a multipartition in $\mathscr{R}$. Set 
$$n_i(\la):=\sharp\{m\mid \la^{(m)}\neq\emptyset\ \text{and}\ k_m=i\}.$$

\vspace{1mm}
\noindent
{\em Claim 2.} Let $\Theta\subseteq\mathscr{R}$ be the subset of all multipartitions $\la\in\mathscr{R}$ such that whenever $\la^{(m)}\neq \emptyset$, then $\la^{(a)}\neq \emptyset$ for all $a>m$ with $k_a=k_m$. Then 
$$
\deg_n(\dim_q R^\Omega_\al)=\deg_n\Big(\sum_{\la \in \Theta} \Big(\sum_{\Stab \in \StTab(\la)} q^{\deg(\Stab)}\Big)^2\prod_{i\in I}l_{n_i(\la)}\Big).
$$

\noindent
{\em Proof of Claim 2.} Let $\la\in \mathscr{R}$ and let $\la^+\in\Theta$ be the multipartition obtained by shifting the non-empty components $\la^{(m)}$ of $\la$ corresponding to $m$ with the same $k_m$ (without changing the order of the non-empty components). To be more precise, for each $i$ with $b_i\neq 0$, each nonempty component $\la^{(m)}$ with $k_m=i$ gets moved to a larger position $m+\ga_i(m)$. Note that $\ga_i(m)\leq \ga_i(m')$ whenever $m>m'$ with $k_m=k_{m'}=i$ and $\la^{(m)},\la^{(m')}$ are non-empty. This defines a multipartition $\ga=(\ga_i)_{i\in I}$, where each partition $\ga_i$ has at most $n_i(\la)$ parts. For $\Stab\in\T(\la)$ let $\Stab^+\in\T(\la^+)$ be the corresponding tableau obtained from $\Stab$ by the same shift which takes $\la$ to $\la^+$. Note that $\deg(\Stab^+)=\deg(S)\sum_{i\in I}|\ga_i|$. Let $p_n(t)$ be the generating function for the partitions with at most $n$ parts. Note that $l_{n_i(\la)}=p_{n_i(\la)}(q^2)$. 
Now Claim 2 follows.

\vspace{1 mm}
We now finish the proof of the proposition. 
Denote $$\Theta(\pi):=\{\la \in \Theta| \pi(\la)=\pi\}.$$ 
For $\la\in\Theta(\pi)$, observe that the group $\Si_\pi$ is naturally a parabolic subgroup of $G:=\times_{i\in I}\Si_{n_i(\la)}$, giving an equality for Poincar\'e polynomials 
$$
P_G(t)=P_{\Si_\pi}(t)\sum_{\la\in\Theta(\pi)}t^{\ell_\la},
$$
where $\ell_\la$ is defined before Proposition~\ref{PThu}. This implies
$$l_\pi=\sum_{\la\in\Theta(\pi)}q^{2\ell_\la}\prod_{i\in I}l_{n_i(\la)}.
$$
Now, using Proposition~\ref{PThu}, we have
\begin{align*}
\sum_{\la \in \Theta} \Big(\sum_{\Stab \in \StTab(\la)} q^{\deg(\Stab)}\Big)^2\prod_{i\in I}l_{n_i(\la)}&=
\sum_{\la \in \Theta} \Big(\dim_q S^{\la^{\op}}\Big)^2\prod_{i\in I}l_{n_i(\la)}
\\
&=\sum_{\la \in \Theta} \Big(\dim_q \bar\De(\pi(\la))\langle\ell_\la\rangle\Big)^2\prod_{i\in I}l_{n_i(\la)}
\\
&=\sum_{\la \in \Theta} q^{2\ell_\la}c_\pi^2\prod_{i\in I}l_{n_i(\la)}
\\
&=\sum_{\pi\in\Pi(\al)}\ \sum_{\la\in\Theta(\pi)} q^{2\ell_\la}c_\pi^2\prod_{i\in I}l_{n_i(\la)}
\\
&=\sum_{\pi\in\Pi(\al)}\ c_\pi^2 l_\pi,
\end{align*}
as desired.
\end{proof}

\section{Affine nilHecke algebra}\label{SNH}
In this section we will review mostly well-known facts about the nilHecke algebra,  and obtain a special case of our main result for this algebra. This special case will be needed in the proofs of the general case. 

\subsection{Definition and basic properties}
We denote $a^{th}$ nilHecke algebra by $H_a$. That is, $H_a$ is the associative, unital $(\Z$-)algebra generated by $\{y_1, \dots, y_a, \psi_1, \dots, \psi_{a-1}\}$ subject to the relations
\begin{align}
  \psi_{r}^2 &= 0
   \label{eq:HeckeRel1}
   \\
  \psi_{r} \psi_{s} &= \psi_{s} \psi_{r} \qquad \textup{if $|r-s| > 1$}
   \label{eq:HeckeRel2}
   \\
  \psi_{r} \psi_{r+1} \psi_{r} &= \psi_{r+1} \psi_{r} \psi_{r+1}
   \label{eq:HeckeRel3}
  \\
  \psi_{r} y_{s} &= y_{s} \psi_{r} \qquad \textup{if $s\neq r,r+1$}
   \label{eq:HeckeRel4}
  \\
  \psi_{r} y_{r+1} &= y_{r} \psi_{r} + 1
   \label{eq:HeckeRel5}
  \\
  y_{r+1} \psi_{r} &= \psi_{r} y_{r} + 1.
   \label{eq:HeckeRel6}
\end{align}
  For $w \in \Si_a$, pick any reduced decomposition $w = s_{i_1} \dots s_{i_k}$. We define $\psi_w = \psi_{i_1} \dots \psi_{i_k}$. In view of the relations above, $\psi_w$ does not depend on the choice of reduced decomposition.
We define $\deg(y_r) = 2$ and $\deg(\psi_r) = -2$; this turns $H_a$ into a graded algebra. 

 There is an involutive homogeneous degree zero anti-automorphism $\tau$ of  $H_a$ fixing the standard generators of $H_a$. We write $h^\tau$ instead of $\tau(h)$ for $h\in H_a$. 
 Given a (graded) left $H_a$-module $M$, we write $M^\tau$ for the (graded) right $H_a$-module given by twisting with $\tau$.
The following result gives {\em standard bases} of $H_a$:

\begin{Theorem} \label{TBasis} 
We have
\begin{enumerate}
\item[{\rm (i)}] $\{\psi_w y_1^{m_1} \dots y_a^{m_a}\mid w \in \Si_a,\ m_1,\dots,m_a \geq 0\}$ is a $\Z$-basis of $H_a$.
\label{zzPBW1}
\item[{\rm (ii)}] $\{y_1^{m_1} \dots y_a^{m_a}\psi_w \mid w \in \Si_a,\ m_1,\dots,m_a \geq 0\}$ is a $\Z$-basis of $H_a$. 
\label{zzPBW2}
\end{enumerate}
In particular,
$$\dim_q(H_a)=\frac{1}{(1-q^2)^a} \sum_{w \in \Si_a} q^{\deg(\psi_w)}.
$$
\end{Theorem}

In view of the theorem we can consider the polynomial algebra 
$$\Pol_a:=\Z[y_1, \dots, y_a]$$ 
as a subalgebra of $H_a$. 
Moreover, let
$$
\Sym_a:=\Z[y_1, \dots, y_a]^{\Si_a}
$$
be the algebra of symmetric functions. The following is well-known, see e.g. \cite{Manivel}. 

\begin{Theorem} \label{TCenter} 
The center of $H_a$ is given by $Z(H_a) = \Sym_a$. 
\end{Theorem}

\subsection{\boldmath The idempotent $e_a$}
It is well-known that $H_a$ can be realized as  the subalgebra of the endomorphism algebra $\End_\Z(\Pol_a)$ generated by (multiplication by) each $y_r$, and the divided difference operators
\begin{equation}\label{ENewton}
  \psi_r(f) = \frac{f - s_r f}{y_{r+1} - y_r},
\end{equation}
  where $(s_r (f))(y_1, \dots, y_a) = f(y_1, \dots, y_{r+1}, y_r, \dots, y_a)$. In light of this description there is an $H_a$-module structure on $\Pol_a$; we shall refer to this module also as $\Pol_a$. 

  Let 
  $$\de_a = y_2 y_3^2 \dots y_a^{a-1},$$ 
  and define $w_0 \in \Si_a$ to be the longest element. It is noticed in \cite[Section 2.2]{KL1} that 
  $$e_a := \psi_{w_0} \de_a$$ 
  is an idempotent. Then $\psi_{w_0}\de_a\psi_{w_0}\de_a=\psi_{w_0}\de_a$ implies 
 \begin{equation}\label{EZD}
  e_a \psi_{w_0} = \psi_{w_0},
 \end{equation}
since by Theorem~\ref{TBasis}(i), $\de_a$ is not a zero divisor. We will need the following facts coming from the theory of Schubert polynomials, see e.g. \cite[Section 10.4]{Fulton}.

\begin{Theorem}\label{thm:polySym}
$\Pol_a$ is a free $\Sym_a$-module with basis $\{\psi_w(\de_a)\ |\ w \in \Si_a\}$. Moreover, $\psi_{w_0}(\de_a)=1$. 
\end{Theorem}

The following two theorems are known, but we sketch their proofs for the reader's convenience.

\begin{Theorem}\label{thm:nHFacts}
  The following things are true: 
\begin{enumerate}
  \item[{\rm (i)}] $_{H_a}\!\Pol_a \iso H_a e_a$,  $f \mapsto f e_a$.\label{zzHe}
  \item[{\rm (ii)}] $(\Pol_a^\tau)_{H_a} \iso e_a H_a$,  $f \mapsto e_a \psi_{w_0} f$.\label{zzeH}
  \item[{\rm (iii)}] $\Sym_a \iso e_a H_a e_a$, $f \mapsto f e_a$.\label{zzeHe}
  \end{enumerate}
\end{Theorem}

\begin{proof}
  (i) By Theorem~\ref{TBasis}(ii), $H_a e_a$ is spanned by elements of the form $y_1^{m_1} \dots y_a^{m_a} \psi_w e_a$. But $\psi_w \psi_{w_0} = 0$ whenever $w \neq 1$, so $H_a e_a$ is in fact spanned by elements of the form $y_1^{m_1} \dots y_a^{m_a} e_a$. By (\ref{EZD}), we have 
$$
y_1^{m_1} \dots y_a^{m_a} e_a\psi_{w_0}=y_1^{m_1} \dots y_a^{m_a} \psi_{w_0}. 
$$  
Since such elements are linearly independent, our spanning set above is actually a basis. In particular, the map
$$
    \Pol_a \to H_a e_a, \ 
    f   \mapsto f e_a
    $$
  is an isomorphism of $\Z$-modules. To show that it is $H_a$-equivariant, note that  the action of $y_r$ is preserved, and furthermore 
  $$\psi_r f e_a = f \psi_r e_a + \psi_r(f) e_a = \psi_r(f) e_a,
  $$
  where $\psi_r(f)$ is the action on $\Pol_a$ defined in (\ref{ENewton}).

  (ii) is proved similarly to (i).

  (iii) $e_a H_a e_a$ is spanned by the elements $e_a f e_a$ with $f \in \Pol_a$. Using (i) we get
  \[
    e_a f e_a = \psi_{w_0} \de_a f e_a = \psi_{w_0}(\de_a f) e_a,
  \]
  and $\psi_{w_0}(\de_a f) \in \Sym_a$. We thus see that $e_a H_a e_a$ is spanned by $b e_a$ with $b \in \Sym_a$. Rewrite again: $b e_a = e_a b = \psi_{w_0} \de_a b$. Now, by Theorem~\ref{TBasis}(i), 
  \begin{align*}
   \Sym_a \to e_a H_a e_a,\  
    b \mapsto b e_a
  \end{align*}
  is an isomorphism.
\end{proof}

\begin{Theorem} 
Let $\iota:H_a\to \End_\Z(\Pol_a)$ be the map which comes from the action of $H_a$ on $\Pol_a$. This map yields an isomorphism of algebras 
$$
\iota:H_a\iso \End_{\Sym_a}(\Pol_a).
$$
\end{Theorem}
\begin{proof}
Let $x=\sum_u f_u\psi_u\in H_a$ be a non-zero element. Let $u$ be a minimal element in the Bruhat order with $f_u\neq 0$. Apply $x$ to the element $\psi_{u^{-1}w_0}(\de_a)\in\Pol_a$, see Theorem~\ref{thm:polySym}. Then $x(\psi_{u^{-1}w_0}(\de_a))=f_u$, which shows that $\iota$ is injective. 

On the other hand, since $\Sym_a=Z(H_a)$, it is clear that the image of $\iota$ is contained in $\End_{\Sym_a}(\Pol_a)$. Using Theorem~\ref{thm:polySym} again and comparing the graded dimensions, we see that $\iota$ is an isomorphism. 
\end{proof}

\begin{Corollary} \label{CSurj}
We have $H_ae_aH_a=H_a$.
\end{Corollary}
\begin{proof}
By using the basis of Theorem~\ref{thm:polySym}, ordered so that $\de_a$ is the first element, we can identify 
$\End_{\Sym_a}(\Pol_a)$ with the matrix algebra $M_{n!}(\Sym_a)$. Then 
under the isomorphism $\iota$ from the theorem $\tau(e_a)$ gets mapped to the matrix unit $E_{1,1}$. The result follows. 
\end{proof}

\subsection{Affine cellular basis of the nilHecke algebra}

\begin{Lemma}\label{cor:nHFree}
We have:
\begin{enumerate}
\item[{\rm (i)}]   $H_a e_a$ is free as a right $e_a H_a e_a$-module with basis $\{\psi_w \de_a e_a\ | w \in \Si_a\}$;
\item[{\rm (ii)}] $e_a H_a$ is free as a left $e_a H_a e_a$-module with basis $\{e_a \psi_v^\tau\ | v \in \Si_a\}$.
\end{enumerate}
\end{Lemma}
\begin{proof}
By Theorem~\ref{thm:polySym}, a basis for $\Pol_a$ over $\Sym_a$ is given by all $\psi_w(\de_a)$ for $w \in \Si_a$. Now by Theorem~\ref{thm:nHFacts}(i),(iii), $H_a e_a$ is free as a right $e_a H_a e_a$-module with basis $\{\psi_w(\de_a) e_a\ | w \in \Si_a\}$. But $\psi_w(\de_a) e_a=\psi_w\de_ae_a$ by Theorem~\ref{thm:nHFacts}(i) again. 

For (ii), we use Theorem~\ref{thm:nHFacts}(ii),(iii) instead to conclude that the set $\{e_a\psi_{w_0}\psi_w(\de_a)\mid w\in\Si_a\}$ is a basis of $e_a H_a$ as a left $e_a H_a e_a$-module. Notice that 
$$
e_a\psi_{w_0}\psi_w(\de_a)=e_a\psi_{w_0}\de_a\psi_w^\tau=e_a\psi_{w}^\tau,
$$
and the result follows. 
\end{proof}

The following theorem gives an affine cellular basis of $H_a$. 


\begin{Theorem}\label{thm:nHCell}
  Let $\{b_x\}_{x \in X}$ be any $\Z$-basis of $\Sym_a$. The nilHecke algebra $H_a$ has a basis given by $\{\psi_w b_x \de_a e_a \psi_v^\tau\ |\ v, w \in \Si_a, x \in X\}$.
\end{Theorem}

\begin{proof}
By Lemma~\ref{cor:nHFree}, the image $H_ae_aH_a$ of the multiplication map 
  \[
    H_a e_a \otimes_{e_a H_a e_a} e_a H_a \to H_a e_a H_a
  \]
is spanned by the set $\{\psi_w b_x \de_a e_a \psi_v^\tau\ |\ v, w \in \Si_a, x \in X\}$. By Corollary~\ref{CSurj}, this set thus spans $H_a$. 

  Next, we compute the degree $d$ of each element of this spanning set, add up the various $q^d$, and see that this is exactly the graded dimension of $H_a$. This shows that this spanning set must be a basis.

  The degree of $\psi_w$ is $-2 \ell(w)$; the degree of $\de_a$ is $a(a-1)$; the degree of $e_a$ is 0. The graded dimension of $\Sym_a$ is $\prod_{r=1}^{a} \frac{1}{1-q^{2r}}$. Let $\{b_x\}_{x\in X}$ be a homogeneous basis of $\Sym_a$ (for example, the monomial symmetric functions). Then
  \begin{align*}
    \sum_{v, w \in \Si_a, x \in X}& q^{\deg(\psi_w) + \deg(b_x) + \deg(\de_a) + \deg(e_a) + \deg(\psi_v)} \\
      &= \left(\sum_{w \in \Si_a} q^{-2 \ell(w)} \right)
         \left(\prod_{r=1}^{a} \frac{1}{1-q^{2r}} \right)
         q^{a(a-1)}
         \left(\sum_{v \in \Si_a} q^{\deg(\psi_v)}\right) \\
      &= \left(q^{-a(a-1)} \prod_{r=1}^a \frac{1-q^{2r}}{1-q^2}\right)
         \left(\prod_{r=1}^{a} \frac{1}{1-q^{2r}} \right)
         q^{a(a-1)} 
         \left(\sum_{v \in \Si_a} q^{\deg(\psi_v)}\right) \\
      &= \left(\frac{1}{(1-q^2)^a} \right) \left(\sum_{v \in \Si_a} q^{\deg(\psi_v)}\right),
  \end{align*}
  which is $\dim_q(H_a)$ by Theorem~\ref{TBasis}, and we are done.
  \end{proof}

\section{Affine cellular structure}\label{SMain}
Throughout this section we work with a fixed element $\al\in Q_+$ of height~$d$. 

\subsection{Basic definitions}\label{SSBasic}
Let $\al^1,\dots,\al^l$ be elements of $Q_+$ with $\al^1+\dots+\al^l=\al$. Then we have a natural embedding 
$$
\iota_{\al^1,\dots,\al^l}:R_{\al^1}\otimes\dots\otimes R_{\al^l}\into R_\al
$$
of algebras, whose image is the parabolic subalgebra $R_{\al^1,\dots,\al^l}$.  

Define the element $\psi_\al\in R_{2\al}$ to be
$$
\psi_\al:=(\psi_d\dots \psi_{2d-1})\dots (\psi_2\dots\psi_{d+1})(\psi_1\dots\psi_d).
$$
In other words, $\psi_\al$ is a `permutation of two $\al$-blocks' and corresponding to the following element of $\Si_{2d}$:
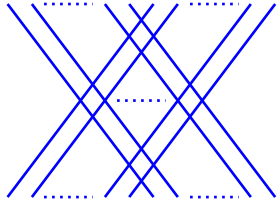
$$
\begin{braid}\tikzset{scale=0.8,baseline=12mm}
    \draw(1,8)--(7,0);
    \draw(2,8)--(8,0);
    \draw[dotted](2.5,8)--(4.5,8);
    \draw[dotted](5.5,4)--(7.5,4);
    \draw[dotted](2.5,0)--(4.5,0);
    \draw(5,8)--(11,0);
    \draw(6,8)--(12,0);
    \draw(7,8)--(1,0);
    \draw(8,8)--(2,0);
    \draw[dotted](8.5,8)--(10.5,8);
    \draw[dotted](8.5,0)--(10.5,0);
    \draw(11,8)--(5,0);
    \draw(12,8)--(6,0);
  \end{braid}
  $$

Now, let $p\in\Z_{>0}$. We define 
$$
\psi_{\al,r}:=\iota_{(r-1)\al,2\al,(p-r-1)\al}(1\otimes \psi_\al\otimes 1)\in R_{p\al}\qquad(1\leq r<p).
$$
In other words, $\psi_{\al,r}$ is a `permutation of the $r^{th}$ and $(r+1)^{st}$ $\al$-blocks'. Moreover, let $w\in\Si_p$ with reduced decomposition $w=s_{i_1}\dots s_{i_m}$. Define an element 
$$
\psi_{\al,w}:=\psi_{\al,i_1}\dots\psi_{\al,i_m}\in R_{p\al}. 
$$

Let also
$$
y_{\al,s}:=\iota_{(s-1)\al,\al,(p-s)\al}(1\otimes y_d\otimes 1)\in R_{p\al}\qquad(1\leq s\leq p).
$$
In other words, $y_{\al,s}$ is a `dot on the last strand of the $s^{th}$ block of size $d$'.
Further, define
$$
\de_{\al,p}:=y_{\al,2}y_{\al,3}^2\dots y_{\al,p}^{p-1}\in R_{p\al}.
$$
We have polynomial algebra and the symmetric polynomial algebra 
$$
\Pol_{\al,p}=\Z[y_{\al,1},\dots,y_{\al,p}]\quad\text{and}\quad
\Sym_{\al,p}=\Pol_{\al,p}^{\Si_p}.
$$

Now, let 
$$\pi = \beta_1^{p_1} \dots \beta_N^{p_N}\in\Pi(\al)$$ 
be a root partition.  
For $1\leq k\leq N$ and $x\in R_{p_k\be_k}$, we put 
$$
\iota^k(x):=\iota_{p_1\be_1+\dots+p_{k-1}\beta_{k-1},p_k\be_k,p_{k+1}\be_{k+1}+\dots+p_N\be_N}(1\otimes x\otimes 1)\in R_\al. 
$$

Define for all $1\leq k\leq N$, $w\in\Si_{p_k}$, $1\leq r<p_k$ and $1\leq s\leq p_k$, the elements of $R_\al$: 
$$\psi_{k,w}:=\iota^k(\psi_{\be_k,w}),\quad 
\psi_{k, r}:=\iota^k(\psi_{\be_k,r}),\quad
y_{k, s}:=\iota^k(y_{\be_k,s}).$$ 
In other words, $\psi_{k, r}$ is
 the permutation of the $r^{th}$ and $(r+1)^{st}$ $\beta_k$-blocks, and  $y_{k, s}$ is a dot on the final strand in the $s^{th}$ $\be_k$-block.

Finally, define 
\begin{align*}
y_\pi &= \iota^1(\de_{\be_1,p_1})\dots  \iota^N(\de_{\be_N,p_N}),
\\
\psi_\pi &= \iota^1(\psi_{\be_1, w^1_0}) \dots \iota^N(\psi_{\be_N, w^N_0}),
\end{align*}
where $w_0^k$ is the longest element of $\Si_{p_k}$ for $k=1,\dots,N$. 
We always make a choice of a reduced decompositon of $w_0^k$ which is left-right symmetric in the sense of Lemma~\ref{LLR}. This will guarantee that 
\begin{equation}\label{ETauInv}
\psi_\pi^\tau=\psi_\pi.
\end{equation}

Recalling the dominant word $\bi_\pi\in\words_\al$, put 
\begin{align*}
e_\pi =  \psi_\pi y_\pi e(\bi_\pi).
\end{align*}
Also, let
\begin{equation}
\label{ELaPi}
\Sym_\pi = \iota_{p_1\be_1,\dots,p_N\be_N}(\Sym_{\be_1,p_1}\otimes\dots\otimes \Sym_{\be_N,p_N})\cong \Sym_{p_1}\otimes\dots\otimes \Sym_{p_N}.
\end{equation}

\subsection{Cells} 

Define 
\begin{align*}
I_\pi' &= \Z\text{-}\spa\{\psi_w y_\pi \Sym_\pi e_\pi  \psi_v^\tau \ |\ w, v \in \Si^\pi\}
\\
I_\pi &= \sum_{\si \geq \pi}{I_\si'}\\
I_{>\pi} &= \sum_{\si > \pi}{I_\si'}
\end{align*}
It will turn out that the $I_\pi$ for $\pi\in \Pi(\al)$ form a chain of cell ideals.

\begin{Lemma}\label{lem:ribbonsOnly}
  Let $w \in \Si^\pi$. Then $w\cdot \bi_\pi=\bi_\pi$ if and only if $w$ permutes the $\pi$-blocks of weight $\be_k$ for all $k=1,\dots,N$.  
\end{Lemma}

\begin{proof}
  This follows from \cite[Lemma 5.3(ii)]{KR}.
\end{proof}

\begin{Lemma} \label{LComm}
$\psi_\pi e(\bi_\pi)$ and $e_\pi$ commute with elements of $\Sym_\pi$. 
\end{Lemma}
\begin{proof}
It suffices to observe that any $\psi_{k,r}e(\bi_\pi)$ commutes with elements of $\Sym_\pi$, which easily follows from the relations in $R_\al$. 
\end{proof}

\begin{Lemma} \label{LITauInv}
We have $\tau(I_\pi')=I_\pi'$ and $\tau(I_\pi)=I_\pi$.
\end{Lemma}
\begin{proof}
It suffices to prove the first equality. 
Using the definition of $e_\pi$ and Lemma~\ref{LComm}, we get
\begin{align*}
\tau(\psi_w y_\pi \Sym_\pi e_\pi  \psi_v^\tau)&
=\tau(\psi_w y_\pi  \Sym_\pi \psi_\pi y_\pi e(\bi_\pi)  \psi_v^\tau)\\
&=\tau(\psi_w y_\pi  \psi_\pi \Sym_\pi y_\pi e(\bi_\pi)  \psi_v^\tau)\\
&=\psi_v   y_\pi e(\bi_\pi) \Sym_\pi\psi_\pi^\tau  y_\pi \psi_w^\tau\\
&=\psi_v   y_\pi  \Sym_\pi\psi_\pi^\tau  y_\pi e(\bi_\pi)\psi_w^\tau.
\end{align*}
It suffices to note that $\psi_\pi^\tau=\psi_\pi$ by (\ref{ETauInv}). 
\end{proof}


\subsection{Ideal filtration}
This subsection is devoted to the proof of the following theorem. 

\begin{Theorem}\label{thm:Ideal}
  $I_\pi$ is the two-sided ideal $\sum_{\si \geq \pi}R_\al e(\bi_\si) R_\al$.
\end{Theorem}

We prove the theorem by downward induction on the lexicographic order on $\Pi(\al)$. To be more precise, throughout the subsection we assume that we have proved that 
\begin{equation}\label{EInduction}
I_{>\pi}=\sum_{\si > \pi}R_\al e(\bi_\si) R_\al
\end{equation}
and from this prove that $I_\pi=\sum_{\si \geq \pi}R_\al e(\bi_\si) R_\al$. When $\pi$ is the maximal root partition, the inductive assumption is trivially satisfied. Otherwise, $I_{>\pi} = I_{\si}$ where $\si$ is the immediate successor of $\pi$ in the lexicographic order.

\begin{Lemma}\label{lem:bigWords}
  If $\bi > \bi_\pi$, then $e(\bi) \in I_{>\pi}$.
\end{Lemma}

\begin{proof}
If $\pi$ is the maximal root partition, then  there is nothing to prove.

  Let $I$ be any maximal (graded) left ideal containing $I_{>\pi}$. Then $R_\al / I \cong L(\si)$ for some $\si$. If $\si > \pi$, then $e(\bi_\si) \in I_{>\pi}$ by induction, see (\ref{EInduction}), and since $e(\bi_\si) L(\si) \neq 0$ we would have $I L(\si)=I(R_\al/I) \neq 0$, which is a contradiction. We conclude that $\si\leq \pi$. 
  
  Therefore, since $\bi > \bi_\pi \geq \bi_\si$ and since all of the weights appearing in $L(\si)$ are less than or equal to $\bi_\si$, we have $e(\bi) L(\si) = 0$, which implies that $e(\bi) \in I$. We have shown that $e(\bi)$ is contained in every maximal left ideal containing $I_{>\pi}$.

Consider the graded left ideal $J:=I_{>\pi} + R_\al(1-e(\bi))$. If $J$ is not all of $R_\al$, then it is contained in a maximal left ideal $I$, which by the previous paragraph contains $e(\bi)$. Since $1-e(\bi)\in J\subseteq I$, we conclude that $I=R_\al$, which is a contradiction. Therefore $J=R_\al$, and we may write $1 = x + r(1-e(\bi))$ for some $x \in I_{>\pi}$ and $r \in R_\al$. Multiplying on the right by $e(\bi)$, we see that $e(\bi) = x e(\bi) \in I_{>\pi}$.

This argument actually proves the lemma over any field, and then it also follows for $\Z$ by a standard argument. 
\end{proof}

\begin{Corollary}\label{cor:paraSubgp}
If $w \in \Si_\pi \setminus \{1\}$, then $\psi_w \Pol_d e(\bi_\pi)  \subseteq I_{>\pi}.$
\end{Corollary}

\begin{proof}
Observe that $w\cdot \bi_\pi > \bi_\pi$, whence $e(w\cdot\bi_\pi) \in I_{>\pi}$ by Lemma~\ref{lem:bigWords}. Now, for any $f\in \Pol_d$, we have  
$\psi_w f e(\bi_\pi)=e(w\cdot\bi_\pi)\psi_w f e(\bi_\pi)\in I_{>\pi}$. 
\end{proof}

Recall $\pi$-blocks defined in the end of Section~\ref{SSBasic}.

\begin{Corollary}\label{cor:yEqual}
  If $y_r$ and $y_s$ are in the same $\pi$-block, then $$y_r e(\bi_\pi) \equiv y_s e(\bi_\pi)\pmod{I_{>\pi}}.$$
\end{Corollary}

\begin{proof}
If $r$ and $r+1$ are in the same $\pi$-block, then $s_r\in\Si_\pi\setminus\{1\}$, and so 
$(y_r - y_{r+1}) e(\bi_\pi) =\psi_r^2 e(\bi_\pi) \in I_{>\pi}$ by Corollary~\ref{cor:paraSubgp}.
\end{proof}

  Let us make the choice of reduced decompositions in $\Si_d$ so that whenever $w=w^\pi w_\pi$ for $w^\pi\in\Si^\pi$ and $w_\pi\in \Si_\pi$, then we have 
\begin{equation}\label{EChoice}
  \psi_w=\psi_{w^\pi}\psi_{w_\pi}.  
  \end{equation}

  Recall the nilHecke algebra $H_a$ from Section~\ref{SNH}.

\begin{Lemma}\label{lem:nilHecke}
For each $k=1,\dots,N$ there is a ring homomorphism 
  \begin{align*}
  \theta_k: H_{p_k} &\to (e(\bi_\pi) R_\al e(\bi_\pi) + I_{>\pi}) / I_{>\pi},
  \\
    y_s &\mapsto y_{k, s} e(\bi_\pi) + I_{>\pi}\qquad(1\leq s\leq p_k),
    \\
    \psi_r &\mapsto \psi_{k, r} e(\bi_\pi) + I_{>\pi}\qquad(1\leq r<p_k).
  \end{align*}
\end{Lemma}

\begin{proof}
We check the relations (\ref{eq:HeckeRel1})--(\ref{eq:HeckeRel6}). 
The relations (\ref{eq:HeckeRel2}) and (\ref{eq:HeckeRel4}) are obvious.
  
To check relation (\ref{eq:HeckeRel1}), we write $\psi_{k, r}^2 e(\bi_\pi) = \sum_{u \in \Si_d} \psi_u f_u e(\bi_\pi)$ with $f_u \in \Pol_d$. If $u \notin \Si^\pi$, then $\psi_u f_u e(\bi_\pi) \in I_{>\pi}$ by (\ref{EChoice}) and Corollary~\ref{cor:paraSubgp}. On the other hand, suppose that $u \in \Si^\pi$ is such that $f_u \neq 0$. By Lemma~\ref{lem:ribbonsOnly}, $u$ permutes $\pi$-blocks of weight $\be_k$. Since $\psi_{k, r}$ only contains crossings between the $r^{th}$ and $(r+1)^{st}$ $\pi$-blocks of weight $\beta_k$, the same must be true for $\psi_u$. The only possibilities are $\psi_u = 1$ or $\psi_u = \psi_{k, r}$. Since $\deg(\psi_{k,r}^2 e(\bi_\pi)) = -4$, these are ruled out by degree. Thus $\psi_{k, r}^2 e(\bi_\pi)\in I_{>\pi}$. 

To check relation (\ref{eq:HeckeRel3}), we write 
$$(\psi_{k, r}\psi_{k,r+1}\psi_{k,r}-\psi_{k,r+1}\psi_{k,r}\psi_{k,r+1}) e(\bi_\pi) = \sum_{u \in \Si_d} \psi_u f_u e(\bi_\pi)$$ 
with $f_u \in \Pol_d$, and show as above that $\psi_u$ that appear with non-zero $f_u$ must be of the form $1,\psi_{k,r},\psi_{k,r+1},\psi_{k,r}\psi_{k,r+1}$ or $\psi_{k,r+1}\psi_{k,r}$. Since $$\deg(\psi_{k, r}\psi_{k,r+1}\psi_{k,r} e(\bi_\pi)) = -6$$ these are ruled out by degree. 

We finally check relation (\ref{eq:HeckeRel5}); relation (\ref{eq:HeckeRel6}) is checked similarly. Write $\bi_{\beta_k} = (i,i+1, \dots, j)$. We calculate $\psi_{k, r} y_{k, r+1} e(\bi_\pi)$ using Khovanov-Lauda diagram calculus, as explained in Section~\ref{SPrel}. In doing so, we will ignore the strands outside of the $r^{th}$ and $(r+1)^{st}$ $\pi$-blocks of weight $\be_k$. We have:  
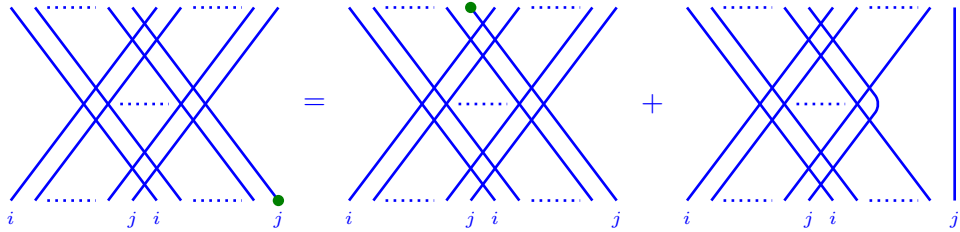
\[
  \begin{braid}\tikzset{scale=0.8,baseline=12mm}
    \draw(1,8)--(7,0)node[below]{$i$};
    \draw(2,8)--(8,0);
    \draw[dotted](2.5,8)--(4.5,8);
    \draw[dotted](5.5,4)--(7.5,4);
    \draw[dotted](2.5,0)--(4.5,0);
    \draw(5,8)--(11,0);
    \draw(6,8)--(12,0)node[below]{$j$};
    \draw(7,8)--(1,0)node[below]{$i$};
    \draw(8,8)--(2,0);
    \draw[dotted](8.5,8)--(10.5,8);
    \draw[dotted](8.5,0)--(10.5,0);
    \draw(11,8)--(5,0);
    \draw(12,8)--(6,0)node[below]{$j$};
    \greendot(12,0);
  \end{braid}
  \begin{braid}\tikzset{scale=0.8,baseline=12mm}
    \draw(0,4)node[font=\normalsize]{$=$};
  \end{braid}
  \begin{braid}\tikzset{scale=0.8,baseline=12mm}
    \draw(1,8)--(7,0)node[below]{$i$};
    \draw(2,8)--(8,0);
    \draw[dotted](2.5,8)--(4.5,8);
    \draw[dotted](5.5,4)--(7.5,4);
    \draw[dotted](2.5,0)--(4.5,0);
    \draw(5,8)--(11,0);
    \draw(6,8)--(12,0)node[below]{$j$};
    \draw(7,8)--(1,0)node[below]{$i$};
    \draw(8,8)--(2,0);
    \draw[dotted](8.5,8)--(10.5,8);
    \draw[dotted](8.5,0)--(10.5,0);
    \draw(11,8)--(5,0);
    \draw(12,8)--(6,0)node[below]{$j$};
    \greendot(6,8);
  \end{braid}
  \begin{braid}\tikzset{scale=0.8,baseline=12mm}
    \draw(0,4)node[font=\normalsize]{$+$};
  \end{braid}
  \begin{braid}\tikzset{scale=0.8,baseline=12mm}
    \draw(1,8)--(7,0)node[below]{$i$};
    \draw(2,8)--(8,0);
    \draw[dotted](2.5,8)--(4.5,8);
    \draw[dotted](5.5,4)--(7.5,4);
    \draw[dotted](2.5,0)--(4.5,0);
    \draw(5,8)--(11,0);
    \draw(6,8)--(9,4)--(6,0)node[below]{$j$};
    \draw(7,8)--(1,0)node[below]{$i$};
    \draw(8,8)--(2,0);
    \draw[dotted](8.5,8)--(10.5,8);
    \draw[dotted](8.5,0)--(10.5,0);
    \draw(11,8)--(5,0);
    \draw(12,8)--(12,0)node[below]{$j$};
  \end{braid}.
\]
  The first term is $y_{k, r} \psi_{k, r} e(\bi_\pi)$. The second term is
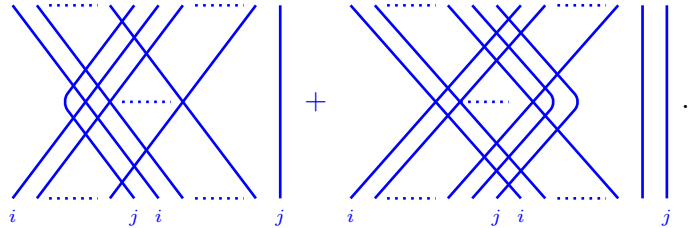
\[
  \begin{braid}\tikzset{scale=0.8,baseline=12mm}
    \draw(1,8)--(7,0)node[below]{$i$};
    \draw(2,8)--(8,0);
    \draw[dotted](2.5,8)--(4.5,8);
    \draw[dotted](5.5,4)--(7.5,4);
    \draw[dotted](2.5,0)--(4.5,0);
    \draw(5,8)--(11,0);
    \draw(6,8)--(3,4)--(6,0)node[below]{$j$};
    \draw(7,8)--(1,0)node[below]{$i$};
    \draw(8,8)--(2,0);
    \draw[dotted](8.5,8)--(10.5,8);
    \draw[dotted](8.5,0)--(10.5,0);
    \draw(11,8)--(5,0);
    \draw(12,8)--(12,0)node[below]{$j$};
  \end{braid}
  \begin{braid}\tikzset{scale=0.8,baseline=12mm}
    \draw(0,4)node[font=\normalsize]{$+$};
  \end{braid}
  \begin{braid}\tikzset{scale=0.8,baseline=12mm}
    \draw(1,8)--(8,0)node[below]{$i$};
    \draw(2,8)--(9,0);
    \draw[dotted](2.5,8)--(4.5,8);
    \draw[dotted](5.5,4)--(7.5,4);
    \draw[dotted](2.5,0)--(4.5,0);
    \draw(5,8)--(12,0);
    \draw(6,8)--(9.5,4)--(6,0);
    \draw(7,8)--(10.5,4)--(7,0)node[below]{$j$};
    \draw(8,8)--(1,0)node[below]{$i$};
    \draw(9,8)--(2,0);
    \draw[dotted](9.5,8)--(11.5,8);
    \draw[dotted](9.5,0)--(11.5,0);
    \draw(12,8)--(5,0);
    \draw(13,8)--(13,0);
    \draw(14,8)--(14,0)node[below]{$j$};
  \end{braid}.
\]
  The first term is in $I_{>\pi}$ by Corollary~\ref{cor:paraSubgp}. Continuing in this way, we obtain the result.
\end{proof}

\begin{Lemma} \label{LNew}
The maps $\theta_1,\dots,\theta_N$ from Lemma~\ref{lem:nilHecke} have commuting images, and so define a map \begin{align*}
\theta: H_{p_1} \otimes \dots \otimes H_{p_N} &\to (e(\bi_\pi) R_\al e(\bi_\pi) + I_{>\pi}) / I_{>\pi},
\\
h_1\otimes\dots\otimes h_N & \mapsto \theta_1(h_1)\dots\theta_N(h_N).
\end{align*}
Moreover, the image of $\theta$ is contained in $I_\pi / I_{>\pi}$.
\end{Lemma}
\begin{proof}
The first statement is obvious. 
To prove the statement about the image of $\theta$, recall the idempotent $e_{p_k} =\psi_{w_0^k} \de_{p_k}\in H_{p_k}$. It is clear that
  \begin{align*}
    \theta(\de_{p_1} \otimes \dots \otimes \de_{p_N}) = y_\pi e(\bi_\pi) + I_{>\pi}\quad \text{and}\quad
    \theta(e_{p_1} \otimes \dots \otimes e_{p_N}) = e_\pi + I_{>\pi}.
  \end{align*}
  For each $k$, choose $w_k, v_k \in \Si_{p_k}$ and $b_k \in \Sym_{p_k}$. Then
  \begin{align*}
    &\theta(\psi_{w_1} b_1 \de_{p_1} e_{p_1} \psi_{v_1}^\tau \otimes \dots \otimes \psi_{w_N} b_N \de_{p_N} e_{p_N} \psi_{v_N}^\tau) \\
         = &(\psi_{1, w_1} \dots \psi_{N, w_N}) (\iota^1(b_1) \dots \iota^N(b_N)) y_\pi e_\pi (\psi_{1, v_1}^\tau \dots \psi_{N, v_N}^\tau) + I_{>\pi}.
  \end{align*}
 The right hand side of this equation is in $I_\pi/I_\pi'$, because $\psi_{1, w_1} \dots \psi_{N, w_N}$ is of the form $\psi_w$ for $w\in\Si^\pi$, $\psi_{1, v_1}^\tau \dots \psi_{N, v_N}^\tau$ is of the form $\psi_v^\tau$ for $v\in \Si^\pi$, and $\iota^1(b_1) \dots \iota^N(b_N)\in \Sym_\pi$. 
 By Theorem~\ref{thm:nHCell}, each $H_{p_k}$ is spanned by $\{\psi_w \Sym_{p_k}\de_{p_k} e_{p_k} \psi_v^\tau\ |\ v, w \in \Si_{p_k}\}$. We conclude that $\im(\theta) \subseteq I_\pi/I_{>\pi}$.
\end{proof}

\begin{Corollary}
  $e(\bi_\pi) \in I_{\pi}$.
\end{Corollary}

\begin{proof}
Note that  $e(\bi_\pi) + I_{>\pi}=\theta(1 \otimes \dots \otimes 1)  \in I_\pi / I_{>\pi}$ by Lemma~\ref{LNew}.
\end{proof}

We come now to the main lemma.

\begin{Lemma}\label{lem:Main}
 If $w \in \Si_d$ and $v \in \Si^\pi$ then $\psi_w \Pol_d e_\pi \psi_v^\tau \subseteq I_\pi$.
\end{Lemma}

\begin{proof}
  We prove this by upward induction on $\deg(\psi_w e(\bi_\pi))$, using as the induction base sufficiently negative degree for which $\psi_w e(\bi_\pi)=0$, and so also $\psi_w \Pol_d e_\pi \psi_v^\tau = \{0\} \subseteq I_\pi$.  

  Let $f \in \Pol_d$. 
  By Corollary~\ref{cor:yEqual}, $f e(\bi_\pi) \equiv g e(\bi_\pi)\pmod{I_{>\pi}}$ for some 
  $$g \in \operatorname{Pol}:=\Z[y_{1,1}, \dots, y_{1,p_1},\dots,y_{N,1},\dots,y_{N,p_N}].$$ Furthermore, we may assume that $g = g_1 \dots g_N$ with each 
  $$g_k \in \operatorname{Pol}_k:=\Z[y_{k,1}, \dots, y_{k, p_k}],$$ since 
  $\operatorname{Pol}\cong\operatorname{Pol}_1\otimes\dots\otimes\operatorname{Pol}_N$.

  Denote by $\widehat{g_k} \in H_{p_k}$ the image of $g_k$ under the isomorphism $\operatorname{Pol}_k \iso \Pol_{p_k}$. Then
  \begin{equation}\label{eq:alpha1}
    \theta(\widehat{g_1}e_{p_1} \otimes \dots \otimes \widehat{g_N}e_{p_N}) = g_1 \dots g_N e_\pi + I_{>\pi}.
  \end{equation}
  On the other hand, by Theorem~\ref{thm:polySym}, there exist $\widehat{b_{k,u}} \in \Sym_{p_k}$ such that
  \begin{align*}
    \widehat{g_k} e_{p_k} &= \sum_{u \in \Si_{p_k}} {\widehat{b_{k,u}} \psi_u (\de_{p_k})} e_{p_k}
            = \sum_{u \in \Si_{p_k}} {\widehat{b_{k,u}} \psi_u \de_{p_k} e_{p_k}}
            = \sum_{u \in \Si_{p_k}} {\psi_u \widehat{b_{k,u}}} \de_{p_k} e_{p_k},
  \end{align*}
where we have used Theorem~\ref{thm:nHFacts}(i) for the second equality and Theorem~\ref{TCenter} for the third equality.   Therefore, denoting
$$
\psi_{u_1,\dots,u_N}:=\psi_{1,u_1}\dots\psi_{N,u_N}\quad\text{and}\quad 
b_{u_1,\dots,u_N}:=b_{1,u_1}\dots b_{N,u_N}
$$
  \begin{equation}\label{eq:alpha2}
    \theta(\widehat{g_1}e_{p_1} \otimes \dots \otimes \widehat{g_N}e_{p_N}) = \sum_{u_1 \in \Si_{p_1},\dots,u_N\in\Si_{p_N}} \psi_{u_1,\dots,u_N}b_{u_1,\dots,u_N} y_\pi e_\pi + I_{>\pi}.
  \end{equation}
  We now equate the right hand sides of (\ref{eq:alpha1}) and (\ref{eq:alpha2}) and multiply by $\psi_w$ on the left and by $\psi_v^\tau$ on the right to obtain
  \[
    \psi_w g e_\pi \psi_v^\tau \equiv \sum_{u_1 \in \Si_{p_1},\dots,u_N\in\Si_{p_N}} \psi_w\psi_{u_1,\dots,u_N}b_{u_1,\dots,u_N} y_\pi e_\pi\psi_v^\tau \pmod{I_{>\pi}}.
  \]
  We are therefore reduced to proving that each summand
 on the right hand side belongs to $I_\pi$. Write, using Theorem~\ref{TBasisGen},
  \[
   \psi_w\psi_{u_1,\dots,u_N} e(\bi_\pi) = \sum_{x \in \Si_d} \psi_{x} f_{x} e(\bi_\pi)
  \]
  so that 
  \begin{equation}\label{eq:rewrite}
    \psi_w\psi_{u_1,\dots,u_N}b_{u_1,\dots,u_N} y_\pi e_\pi\psi_v^\tau = \sum_{x \in \Si_d} \psi_{x} f_{x} b_{u_1,\dots,u_N} y_\pi e_\pi\psi_v^\tau.
  \end{equation}
 
If at least one of $u_1,\dots,u_N$ is not $1$, then $\deg(\psi_{u_1,\dots,u_N} e(\bi_\pi)) < 0$, and so 
  $$\deg(\psi_{x} e(\bi_\pi)) \leq \deg(\psi_{x} f_{x} e(\bi_\pi)) < \deg(\psi_w e(\bi_\pi)).
  $$ 
 Therefore we are done by induction in this case.

 Now, let  $u_1=\dots=u_N=1$. By (\ref{EChoice}), we can write $\psi_w = \psi_{w_1} \psi_{w_2}$ with $w_1 \in \Si^\pi$ and $w_2 \in \Si_\pi$. If $w_2 \neq 1$, then we are done by Corollary \ref{cor:paraSubgp}. Otherwise, recalling that $b_{u_1,\dots,u_N} \in \Sym_\pi$, we have 
$
    \psi_w b_{u_1,\dots,u_N} y_\pi e_\pi \psi_v^\tau \in I_\pi'
$
  by definition.
\end{proof}

\begin{Corollary}
  $I_\pi$ is an ideal.
\end{Corollary}

\begin{proof}
By Lemma~\ref{LITauInv}, we have $\tau(I_\pi)=I_\pi$. So it is enough to prove that $I_\pi$ is a left ideal. 
  Now, for $x \in R_\al$ and $w,v\in\Si^\pi$, we have 
  \begin{align*}
    x \psi_w y_\pi \Sym_\pi e_\pi\psi_v^\tau &= x \psi_w \Sym_\pi e(\bi_\pi) y_\pi e_\pi \psi_v^\tau \\
           &= \sum_{u \in \Si_d} {\psi_u f_u \Sym_\pi e(\bi_\pi)} y_\pi e_\pi \psi_v^\tau \subseteq I_\pi
  \end{align*}
  by Theorem~\ref{TBasisGen} for the second equality and Lemma~\ref{lem:Main} for the inclusion.
\end{proof}

We can now complete the proof of Theorem~\ref{thm:Ideal}. 
We have already proved that $I_\pi$ is an ideal containing $e(\bi_\pi)$. By definition, $I_\pi$ contains $I_{>\pi}$ which by induction is equal to $\sum_{\si > \pi}R_\al e(\bi_\si) R_\al$. Therefore $I_\pi \supseteq \sum_{\si \geq \pi}R_\al e(\bi_\si) R_\al$.

On the other hand, $I_\pi$ is spanned by $I_{>\pi}$ and elements from 
$$\psi_w y_\pi \Sym_\pi e_\pi  \psi_v^\tau \subseteq R_\al e(\bi_\pi) R_\al\qquad(w, v \in \Si^\pi).
$$ 
Therefore $I_\pi \subseteq \sum_{\si \geq \pi}R_\al e(\bi_\si) R_\al$, and so we have equality.
This concludes the proof of Theorem~\ref{thm:Ideal}. 

\subsection{Affine cellular basis}
The following lemma shows that the cell ideals exhaust the algebra $R_\al$.

\begin{Lemma} \label{LRepeated}
If a two-sided ideal $J$ of $R_\al$ contains all idempotents $e(\bi_\pi)$ with $\pi\in\Pi(\al)$, then $J=R_\al$. In particular, $\sum_{\pi\in \Pi(\al)}I_{\pi}'=R_\al$. 
\end{Lemma}
\begin{proof}
If $J\neq R_\al$, let $I$ be a maximal (graded) left ideal containing $J$. Then $R_\al / I \cong L(\pi)$ for some $\pi$. Then $e(\bi_\pi)L(\pi)\neq 0$, which contradicts the assumption that $e(\bi_\pi)\in J$. This argument proves the lemma over any field, and then it also follows for $\Z$. 
\end{proof}


\begin{Theorem}\label{cl:basis}
For each $\pi\in\Pi(\al)$ pick a $\Z$-basis $\{b_{\pi, x}\}_{x\in X(\pi)}$ of $\Sym_\pi$. Then
\begin{equation}\label{ECellBasis}
\{\psi_w y_\pi b_{\pi, x} e_\pi \psi_v^\tau\ |\ \pi \in \Pi(\al),\ x\in X(\pi),\ v, w \in \Si^\pi\}\end{equation}
is a $\Z$-basis of $R_\al$. In particular, $R_\al = \bigoplus_{\pi \in \Pi(\al)} I_\pi'
$. 
\end{Theorem}
\begin{proof}
The elements of (\ref{ECellBasis}) span $R_\al$ in view of Lemma~\ref{LRepeated}. 
So it remains to apply Proposition~\ref{PDim}. 
\end{proof}

\begin{Theorem}\label{TCellMod}
  Let $\pi\in \Pi(\al)$, and $\bar R_\al := R_\al / I_{>\pi}$. Then
  \begin{enumerate}
  \item[{\rm (i)}] the map $\Sym_\pi \to \bar e_\pi \bar R_\al \bar e_\pi,\ b\mapsto \bar b\bar e_\pi$ is an isomorphism of graded algebras;
  \item[{\rm (ii)}] $\bar R_\al \bar e_\pi$ is a free right $\bar e_\pi \bar R_\al \bar e_\pi$-module with basis $\{\bar \psi_w \bar y_\pi \bar e_\pi\ |\ w \in \Si^\pi\}$;
  \item[{\rm (iii)}] $\bar e_\pi \bar R_\al $ is a free left $\bar e_\pi \bar R_\al \bar e_\pi$-module with basis $\{\bar e_\pi \bar \psi_v^\tau\ |\ v \in \Si^\pi\}$;
  \item[{\rm (iv)}] multiplication provides an isomorphism
      \[
        \bar R_\al \bar e_\pi \otimes_{\bar e_\pi \bar R_\al \bar e_\pi} \bar e_\pi \bar R_\al \iso \bar R_\al \bar e_\pi \bar R_\al;
      \]
\item[{\rm (v)}]  $\bar R_\al \bar e_\pi \bar R_\al = I_\pi / I_{>\pi}$.
  \end{enumerate}
 \end{Theorem}

\begin{proof}
(v) follows from Theorem~\ref{thm:Ideal}.

Now pick a $\Z$-basis $\{b_x\}_{x\in X}$ of $\Sym_\pi$. We next prove that
\begin{equation}\label{ELeft}
\{\bar \psi_w  \bar y_\pi \bar b_x \bar e_\pi\mid w \in \Si^\pi,\ x\in X\}
\end{equation}
is a $\Z$-basis of $\bar R_\al \bar e_\pi$ and 
\begin{equation}\label{ERight}
\{\bar\psi_\pi  \bar y_\pi \bar b_x \bar e_\pi \bar\psi_v^\tau\mid v \in \Si^\pi,\ x\in X\}
\end{equation}
is a $\Z$-basis of $\bar e_\pi \bar R_\al$.

To prove that (\ref{ELeft}) is a basis of $\bar R_\al \bar e_\pi$, note that $\bar R_\al \bar e_\pi \bar R_\al= \bar I_\pi=\bar I_\pi'$ by (v), so $\bar R_\al \bar e_\pi  = \bar I_\pi' \bar e_\pi$. A $\Z$-spanning set for $\bar R_\al \bar e_\pi$ is therefore given by $\bar \psi_w \bar y_\pi \bar b_x  \bar e_\pi \bar \psi_v^\tau  \bar e_\pi$ where $v, w \in \Si^\pi$. Using Lemma~\ref{lem:ribbonsOnly}, we have $\bar e_\pi \bar \psi_v^\tau \bar e_\pi = 0$ unless $v$ is a block permutation. On the other hand, if $v$ is a nontrivial block permutation, then $\bar e_\pi  \bar \psi_v^\tau \bar e_\pi = \bar e_\pi \bar \psi_v^\tau \bar \psi_\pi \bar y_\pi \bar e(\bi_\pi) = 0$ in view of Lemma~\ref{LNew}. We can therefore refine our spanning set to (\ref{ELeft}), which is $\Z$-linearly independent by Theorem~\ref{cl:basis}. 

To prove that (\ref{ERight}) is a basis of $ \bar e_\pi \bar R_\al$, note that by definition we have $\bar e_\pi \bar R_\al = \bar \psi_\pi \bar y_\pi \bar e(\bi_\pi) \bar R_\al \subseteq \bar \psi_\pi \bar e(\bi_\pi) \bar R_\al$. On the other hand, by (\ref{EZD}) and Lemma~\ref{LNew}, we have $\bar\psi_\pi \bar e(\bi_\pi) = \bar e_\pi \bar \psi_\pi \bar e(\bi_\pi)$, hence $\bar \psi_\pi \bar e(\bi_\pi) \bar R_\al \subseteq \bar e_\pi \bar R_\al$. Thus $\bar \psi_\pi \bar e(\bi_\pi) \bar R_\al = \bar e_\pi \bar R_\al$. A spanning set for $\bar \psi_\pi \bar e(\bi_\pi)  \bar R_\al$ is given by 
$$\bar \psi_\pi \bar e(\bi_\pi) \bar\psi_w  \bar y_\pi \bar b_x \bar e_\pi \bar \psi_v^\tau$$ 
for $v, w \in \Si^\pi$ and $x\in X$. As in the previous paragraph, this term is zero for $w \neq 1$. The spanning set thus reduces to the elements of (\ref{ERight}),
which are linearly independent by Theorem~\ref{cl:basis}. 

(i) We have $\bar e_\pi  \bar R_\al  \bar e_\pi = \bar e_\pi \bar R_\al \bigcap \bar R_\al \bar e_\pi$, and both of the sets on the right are spanned by a subset of our basis (\ref{ECellBasis}) of $\bar R_\al$. Hence $\bar e_\pi \bar R_\al \bigcap \bar R_\al \bar e_\pi$ has a basis given by those basis elements in both (\ref{ELeft}) and (\ref{ERight}). This is clearly the set of all $\bar \psi_\pi \bar y_\pi  \bar b_x\bar e_\pi = \bar e_\pi \bar b_x \bar e_\pi = \bar b_x \bar e_\pi$ for $x\in X$.
The map
  \begin{align*}
    \Sym_\pi \to \bar e_\pi \bar R_\al \bar e_\pi\ 
    b     \mapsto \bar e_\pi \bar b \bar e_\pi = \bar b \bar e_\pi
  \end{align*}
  is therefore an isomorphism.

(ii) follows immediately from (i) and the fact that (\ref{ELeft}) is a basis of $\bar R_\al \bar e_\pi$. 

(iii) follows from (i), the fact that (\ref{ERight}) is a basis of $ \bar e_\pi \bar R_\al$, and the equality 
$$ \bar\psi_\pi  \bar y_\pi \bar b_x \bar e_\pi \bar\psi_v^\tau= 
 \bar b_x \bar e_\pi \bar e_\pi \bar\psi_v^\tau= 
  \bar b_x \bar e_\pi \bar\psi_v^\tau.$$

 (iv) follows immediately by considering the bases constructed above.
\end{proof}

Recall the definition of an affine cellular algebra given in the introduction.

\begin{Corollary}
  The algebra $R_\al$ is affine cellular with cell chain given by the ideals $\{I_\pi\mid \pi \in \Pi(\al)\}$.
\end{Corollary}
\begin{proof}
  Fix $\pi \in \Pi(\al)$ and write $\bar R_\al:= R_\al / I_{>\pi}$. We must show that $I_\pi / I_{> \pi}$ is an affine cell ideal of $\bar R_\al$. We verify conditions (i)--(iii) of the definition given in the introduction.

  (i) This follows from Lemma~\ref{LITauInv}.

  (ii) Define $V$ to be the free $\Z$-module on the basis $\{w\mid w \in \Si^\pi\}$, and let $B:= \Sym_\pi$ with $\si$ being the trivial involution of $B$. We define the right $B$-module $\De := V \otimes_\Z B$. Theorem~\ref{TCellMod}(i),(ii) implies that the map 
  \[
    \De \to \bar R_\al \bar e_\pi, w \otimes b \mapsto \bar\psi_w \bar y_\pi \bar e_\pi \bar b
  \]
where $w \in \Si^\pi$, $b \in B$, is an isomorphism of right $B$-modules. We use this to define an $\bar R_\al$-$B$-bimodule structure on $\De$.

  (iii) Let $\De':= B\otimes_\Z V$ be the $B$-$\bar R_\al$-bimodule defined as in (\ref{EAction}). 
By Theorem~\ref{TCellMod}(i),(iii), the map 
  \begin{equation}\label{EDel}
    \De' \to \bar e_\pi \bar R_\al,\ b \otimes w \mapsto \bar b \bar e_\pi \bar \psi_w^\tau
  \end{equation}
  for $w \in \Si^\pi$, $b \in B$, defines an isomorphism of left $B$-modules. 
  
 \vspace{1 mm}
\noindent
  {\em Claim.} The map \eqref{EDel} is an isomorphism of $B$-$\bar R_\al$-bimodules.
  
  \noindent
  {\em Proof of Claim.} Note that under the identifications $\De \simeq \bar R_\al \bar e_\pi$ and $\De' \simeq \bar e_\pi \bar R_\al$ (as $B$-modules), the twist map $\tw^{-1}: \De' \to \De$ becomes the map
  \[
    \eta: \bar e_\pi \bar R_\al \to \bar R_\al \bar e_\pi, \bar b \bar e_\pi \bar \psi_w^\tau \mapsto \bar \psi_w \bar y_\pi \bar e_\pi \bar b.
  \]
  Therefore the claim is equivalent to $\eta(\bar b \bar e_\pi \bar \psi_w^\tau r) = r^\tau \bar \psi_w \bar y_\pi \bar e_\pi \bar b$ for $w \in \Si^\pi$, $b \in B$, and $r \in \bar R_\al$. We can write 
$$
\bar b \bar e_\pi \bar \psi_w^\tau r=\sum_{v \in \Si^\pi} \bar b_v  \bar e_\pi \bar \psi_v^\tau
$$  
for some $b_v\in \Sym_\pi$, and then, using (\ref{ETauInv}), we get 
  \begin{align*}
    \eta(\bar b \bar e_\pi \bar \psi_w^\tau r) &= \eta\Big(\sum_{v \in \Si^\pi} \bar b_v  \bar e_\pi \bar \psi_v^\tau\Big) 
      = \sum_{v \in \Si^\pi} \bar \psi_v \bar y_\pi \bar e_\pi \bar b_v \\
      &= \sum_{v \in \Si^\pi} \bar \psi_v \bar y_\pi \bar \psi_\pi \bar y_\pi \bar e(\bi_\pi) \bar b_v 
      = \sum_{v \in \Si^\pi} \bar \psi_v \bar e(\bi_\pi) \bar y_\pi \bar \psi_\pi \bar b_v  \bar y_\pi\\
      &= \sum_{v \in \Si^\pi} \big(\bar b_v  \bar \psi_\pi \bar y_\pi \bar e(\bi_\pi) \bar \psi_v^\tau \big)^\tau \bar y_\pi
      = \sum_{v \in \Si^\pi} \big(\bar b_v  \bar e_\pi \bar \psi_v^\tau \big)^\tau \bar y_\pi\\
      &= \big(\bar b \bar e_\pi \bar \psi_w^\tau r\big)^\tau \bar y_\pi
      = \big(\bar b \bar \psi_\pi \bar y_\pi \bar e(\bi_\pi) \bar \psi_w^\tau r\big)^\tau \bar y_\pi\\
      &= r^\tau \bar \psi_w \bar e(\bi_\pi) \bar y_\pi \bar \psi_\pi \bar b \bar y_\pi
      = r^\tau \bar \psi_w \bar y_\pi \bar e_\pi \bar b.
  \end{align*}
The proof of the claim is now complete, and so we can identify $\De'$ with $\bar e_\pi \bar R_\al$.
  
  By Theorem~\ref{TCellMod}(iv), the map $\De \otimes_B \De' \to \bar I_\pi$ is an isomorphism. One shows that the diagram in part (iii) of the definition of a cell ideal is commutative using an argument similar to the one in the proof of the above claim.
\end{proof}

\end{document}